\documentclass{amsart}
\usepackage[all]{xy}
\usepackage{amsmath,amssymb}
\usepackage{amsthm}
\usepackage{enumerate}
\usepackage[dvips]{graphicx}
\usepackage{txfonts}
\theoremstyle{plain}
  \newtheorem{theorem}{Theorem}[section]
  \newtheorem{corollary}[theorem]{Corollary}
  \newtheorem{lemma}[theorem]{Lemma}
  \newtheorem{proposition}[theorem]{Proposition}
\newtheorem{conjecture}[theorem]{Conjecture}
  
\theoremstyle{definition}
  \newtheorem{definition}[theorem]{Definition}
  \newtheorem{example}[theorem]{Example}
  \newtheorem{remark}[theorem]{Remark}

\newcommand{\ra}{\longrightarrow}
\def\co{\colon\thinspace}

\usepackage{tikz}

\newcommand{\op}{\mathrm{op}}
\newcommand{\K}{\mathcal{K}}
\newcommand{\TC}{\mathrm{TC}}

\newcommand{\SC}{\mathrm{SC}}
\newcommand{\CC}{\mathrm{CC}}
\newcommand{\cat}{\mathrm{cat}}
\newcommand{\pr}{\mathrm{pr}}

\newcommand{\sd}{\mathrm{sd}}
\newcommand{\calS}{\mathcal{S}}

\newtheorem{maina}{Main Theorem}
  
  \newtheorem{mainb}{Main Theorem}

\begin{document}
\author{Kohei Tanaka}
\address{Institute of Social Sciences, School of Humanities and Social Sciences, Academic Assembly, Shinshu University, Japan.}
\email{tanaka@shinshu-u.ac.jp}
\title {A combinatorial description of topological complexity for finite spaces}
\thanks{This work was supported by JSPS KAKENHI Grant Number JP15K17535 and JP17K14183.}

\maketitle
{\footnotesize 2010 Mathematics Subject Classification : 55P10, 06A07}

{\footnotesize Keywords : topological complexity, finite space, order complex}

\begin{abstract}
This paper presents a combinatorial analog of topological complexity for finite spaces.
We demonstrate that this coincides with the genuine topological complexity of the original finite space, and constitutes an upper bound for the topological complexity of its order complex.
Furthermore, we examine the case of the iterated barycentric subdivision of finite spaces and the relation with the simplicial complexity for simiplicial complexes.
\end{abstract}

\section{Introduction}

The \emph{topological complexity} $\TC(X)$ of a space $X$ is a homotopy invariant that measures how complex the space is. This invariant was introduced by Farber in the study of robotics motion planning \cite{Far03}. Let us recall briefly the definition. We deal only with path--connected spaces throughout this paper.

\begin{definition}
Let $X$ be a space, and let $X^{I}$ be the path space of $X$ consisting of continuous paths $\gamma \co I=[0,1] \to X$. The \emph{path fibration} $p \co X^{I} \to X \times X$ is defined by $p(\gamma)=(\gamma(0),\gamma(1))$.
The \emph{topological complexity} $\TC(X)$ is the smallest non--negative integer $n$ such that there exists an open cover $\{U_{i}\}_{i=1}^{n}$ of $X \times X$ with a continuous section $s_{i} \co U_{i} \to X^{I}$ of the path fibration $p$ for each $i$. If such $n$ does not exist, then we define $\TC(X)=\infty$.
\end{definition}

This is a special case of a {\em sectional category} or {\em Schwarz genus} \cite{Sch61} for the path fibration.
In this paper, we focus on the topological complexity for finite $T_0$ spaces, which are regarded as finite partially ordered sets (posets for short) \cite{Sto66}. We introduce another invariant $\CC(P)$, called the {\em combinatorial complexity} for a finite space $P$ using purely combinatorial terms. Our main aim is to show the following equality.

\begin{maina}[Theorem \ref{TC=CC}]
For any finite space $P$, it holds that $\TC(P) = \CC(P)$.
\end{maina}

This theorem suggests that the topological complexity of a finite space can be calculated using combinatorial methods. 

On the other hand, a finite simplicial complex $\K(P)$ can be associated to any finite space $P$, called the \emph{order complex}. The combinatorial complexity $\CC(P)$ is an upper bound for the topological complexity $\TC(|\K(P)|)$ of the geometric realization $|\K(P)|$ of the order complex. However, it does not give a good estimate of $\TC(|\K (P)|)$.
For example, the minimal finite space model $P$ of a circle $S^1$ yields $\CC(P)=4$, whereas $\TC(|\K (P)|)=\TC(S^1)=2$ (see Example \ref{sphere}). This results from a small number of open sets of the product $P \times P$. To fix the problem we define $\CC^k(P)$, a notion of complexity which involves the $k$th barycentric subdivision of $P$.
This idea is based on Gonz\'alez's simplicial approach to topological complexity for simplicial complexes \cite{Gon}. He introduced an invariant $\SC(K)$ for a finite simplicial complex $K$, called the {\em simplicial complexity}, and proved the equality $\SC(K)=\TC(|K|)$. This paper relates $\CC^{\infty}(P) = \lim_{k \to \infty} \CC^k(P)$ with the simplicial complexity $\SC(\K(P))$ of the order complex of a finite space $P$.

\begin{mainb}[Theorem \ref{CC=SC}]
For any finite space $P$, it holds that $\CC^{\infty}(P) = \SC(\K (P))$.
\end{mainb}

By Gonz\'alez's result, we obtain the equality $\CC^{\infty}(P) = \TC(|\K (P)|)$.
This implies that the topological complexity of $|\K (P)|$ can be described in purely combinatorial terms.

The remainder of this paper is organized as follows.
Section 2 provides the definition of combinatorial complexity for finite spaces, including some basic homotopical properties of finite spaces.
In Section 3, we prove the equality between the combinatorial and topological complexity of a finite space. Section 4 develops the idea of combinatorial complexity using barycentric subdivision. We introduce $\CC^{\infty}(P)$ and show the equality between $\CC^{\infty}(P)$ and $\SC(\K(P))$ for any finite space $P$.


\section{Combinatorial complexity for finite spaces}
This paper focuses on finite topological spaces, i.e. spaces consisting of a finite set of points.
We often consider a finite topological space as a discrete space, because every finite $T_{1}$ space must be discrete. In contrast, finite $T_{0}$ spaces play an important role in homotopy theory for finite complexes. In particular, every finite simplicial complex has the weak homotopy type of a finite $T_{0}$ space. 
On the other hand, $T_0$--Alexandroff spaces are closely related to posets. 
Here, an Alexandroff space is a space such that an arbitrary intersection of open sets is open.
A $T_0$--Alexandroff space is equipped with a partial order $x \leq y$ defined by $x \in U_{y}$, where $U_{y}$ is the smallest open set containing $y$. 
Conversely, a poset is equipped with the Alexandroff topology generated from its ideals.
Here, an ideal of a poset $P$ is a subset $Q$ satisfying $x \in Q$ whenever $x \leq y$ for some $y \in Q$.
From the above viewpoint, we can identify a $T_0$--Alexandroff space with a poset. 
In particular, a finite $T_0$ space can be regarded as a finite poset. Let us simply call this a \emph{finite space}. 

Now, we propose a combinatorial analog of topological complexity for finite spaces.
Let $J_m$ denote the finite space consisting of $m+1$ points with the zigzag order,
\[
0 < 1 > 2 < \cdots > (<) m.
\]
This finite space is called the \emph{finite fence} with length $m$, and behaves as an interval in terms of finite spaces. We refer the readers to \cite{Sto66} and \cite{Bar11} for the homotopy theory of finite spaces.
A finite space $P$ is path--connected if and only if for any $x,y \in P$, there exists $m \geq 0$ and a continuous map $\gamma \co J_{m} \to P$ such that $\gamma(0)=x$ and $\gamma(m)=y$. More generally, two maps $f,g \co P \to Q$ between finite spaces are homotopic if and only if there exists $m \geq 0$ and a continuous map $H \co P \times J_{m} \to Q$ such that $H_0=f$ and $H_m = g$. By the exponential law, this is equivalent to considering a continuous map $H' \co J_m \to Q^P$ such that $H'(0)=f$ and $H'(m)=g$.

\begin{definition}
Let $P$ be a finite space. A \emph{combinatorial path} of $P$ with length $m$ is a continuous map $\gamma \co J_m \to P$. Note that a map between finite spaces is continuous if and only if it is order--preserving (a poset map). For this reason, a combinatorial path is a zigzag sequence in $P$ formed as follows:
\[
x_0 \leq x_1 \geq x_2 \leq \cdots \geq (\leq) x_m.
\]
Let $P^{J_m}$ denote the finite space of combinatorial paths of $P$ with length $m$, equipped with the pointwise order.
For a combinatorial path $\gamma \in P^{J_m}$, the inverse path $\gamma^{-1}$ is defined by the path
\[
\gamma(m) \leq \gamma(m-1) \geq \cdots \geq \gamma(0)
\]
with length $m$ if $m$ is even. When $m$ is odd, $\gamma^{-1}$ is the path
\[
\gamma(m) \leq \gamma(m) \geq \gamma(m-1) \leq \cdots \geq \gamma(0)
\]
with length $m+1$. For two combinatorial paths $\gamma \in P^{J_m}$ and $\delta \in P^{J_\ell}$ with $\gamma(m)=\delta(0)$, the concatenation $\gamma * \delta$ is defined by the path
\[
\gamma(0) \leq \gamma(1) \geq \cdots \geq \gamma(m) = \delta(0) \leq \delta(1) \geq \cdots \geq (\leq) \delta(\ell)
\]
with length $m+\ell$ if $m$ is even. When $m$ is odd,  $\gamma * \delta$ is the path
\[
\gamma(0) \leq \gamma(1) \geq \cdots \leq \gamma(m) \geq \gamma(m) = \delta(0) \leq \delta(1) \geq \cdots \geq (\leq) \delta(\ell).
\]
with length $m+\ell+1$.

\end{definition}

Note that the Alexandroff topology on $P^{J_m}$ coincides with the compact open topology, by Stong's result \cite[Proposition 9]{Sto66}. As an analog of path fibration, it is equipped with the canonical continuous map $q_{m} \co P^{J_m} \to P \times P$ given by $q_{m}(\gamma) = (\gamma(0),\gamma(m))$ for each $m \geq 0$.

\begin{definition}
Let $P$ be a finite space. For $m \geq 0$, define $\CC_{m}(P)$ as the smallest non--negative integer $n$ such that there exists an open cover $\{Q_{i}\}_{i=1}^{n}$ of $P \times P$ with a continuous section $s_{i} \co Q_{i} \to P^{J_m}$ of $q_{m}$ for each $i$. If such $n$ does not exist, then we define $\CC_m(P)=\infty$.
\end{definition}

\begin{lemma}\label{m+1<m}
For any $m \geq 0$ and a finite space $P$, it holds that $\CC_{m+1}(P) \leq \CC_{m}(P)$.
\end{lemma}

\begin{proof}
Assume that $\CC_{m}(P)=n$. Then, there exists an open cover $\{Q_{i}\}_{i=1}^{n}$ of $P \times P$ with a continuous section $s_{i} \co Q_{i} \to P^{J_{m}}$ of $q_{m}$ for each $i$.
The retraction $r \co J_{m+1} \to J_{m}$ sending $m+1$ to $m$ induces a map $r^* \co P^{J_{m}} \to P^{J_{m+1}}$ such that the following diagram commutes:
\[
\xymatrix{
P^{J_{m}} \ar[dr]_{q_{m}} \ar[rr]^{r^*} && P^{J_{m+1}} \ar[dl]^{q_{m+1}} \\
& P \times P.
}
\]
The composition $r^* \circ s_{i} \co Q_{i} \to P^{J_{m+1}}$ is a continuous section of $q_{m+1}$ for each $i$. Thus, $\CC_{m+1}(P) \leq n$.
\end{proof}

The topological complexity is closely related to the Lusternik--Schnirelmann (LS) category of a space.
The LS category $\cat(X)$ of a space $X$ is the smallest non--negative integer $n$ such that $X$ can be covered by $n$ open subspaces that are contractible in $X$. Although one often uses the reduced version, which is one less than the definition above, we use the unreduced version throughout this paper.
Farber proved that the following inequalities hold for any space $X$ (see \cite[Theorem 5]{Far03}):
\[
\cat(X) \leq \TC(X) \leq \cat(X \times X). 
\]
Moreover, the product inequality of LS category implies that $\cat(X \times X) \leq 2 \cat(X)-1$ when $X$ is paracompact and Hausdorff.
Let us consider the case of finite spaces.

\begin{lemma}\label{cat}
For any finite space $P$, it holds that $\CC_{m}(P) \leq \cat(P \times P)$ for sufficiently large $m \geq 0$.
\end{lemma}

\begin{proof}
Assume that $\cat(P \times P)=n$. Then, there exists a contractible open cover $\{Q_{i}\}_{i=1}^{n}$ of $P \times P$.
For each $i$, fix an element $(x_1,x_2) \in Q_{i}$ and a path $\gamma \co J_k \to P$ between $x_1$ and $x_2$, 
and choose a contracting homotopy $H \co Q_i \times J_\ell \to P \times P$ onto $(x_1,x_2)$.
The first and second projections yield two maps, $H_1, H_2 \co Q_i \to P^{J_{\ell}}$,
such that $H_{j}(a_1,a_2)(0)=a_j$ and $H_{j}(a_1,a_2)(\ell)=x_j$ for $j=1,2$. 
A continuous section $s_i \co Q_{i} \to P^{J_{m_i}}$ of $q_{m_i}$ is defined by the concatenation of paths $s_{i}(a_1,a_2)=H_{1}(a_1,a_2) * \gamma * H_2(a_1,a_2)^{-1}$, where $m_i$ is the length of $s_i(a_1,a_2)$. Note that $m_i$ depends only on $k$ and $\ell$, not $a_1$ and $a_2$.
Hence, we obtain a continuous section $Q_i \to P^{J_m}$ of $q_m$ for $m=\max \{m_i\}_{i=1}^{n}$. Thus, $\CC_m(P) \leq n$.
\end{proof}

Lemma \ref{cat} can also be deduced from Theorem \ref{TC=CC} below.
We can easily check that $\cat(P \times P) \leq \cat(P)^2$. Hence, we obtain the following corollary.

\begin{corollary}\label{max}
For any finite space $P$, it holds that $\CC_{m}(P) \leq  \cat(P)^2$ for sufficiently large $m \geq 0$.
\end{corollary}

\begin{remark}\label{remark}
Note that the product formula of LS category
\[
\cat(P \times Q) \leq \cat(P) +\cat(Q) -1
\] 
does \emph{not} hold in general for finite spaces $P,Q$. This requires the spaces to be paracompact Hausdorff spaces (occasionally a paracompact space is defined such that it is always Hausdorff) with a partition of unity for every finite open cover (see \cite[Proposition 2.3]{Jam78}). For this reason, the equality $\cat(P \times P) \leq 2 \cat(P) -1$ does not hold for an arbitrary finite space $P$. For example, let $P$ be the finite space consisting of four points described as the following Hasse diagram:
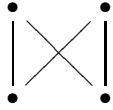
\[
\xymatrix{
\bullet \ar@{-}[d] \ar@{-}[rd]& \bullet \ar@{-}[d] \ar@{-}[ld] \\ 
\bullet & \bullet
}
\]
We can verify that $\cat(P)=2$, however $\cat(P \times P)=4$ (see Corollary \ref{inequalities} and Example \ref{height_1}). 
\end{remark}

Lemma \ref{m+1<m} and Corollary \ref{max} allow us to define the notion of combinatorial complexity for finite spaces.

\begin{definition}
Let $P$ be a finite space.
The \emph{combinatorial complexity} $\CC(P)$ is the minimum of $\CC_m(P)$:
\[
\CC(P) = \min_{m \geq 1} \{\CC_m(P)\} = \lim_{m \to \infty} \CC_m(P) < \infty.
\]
\end{definition}


\section{Topological and combinatorial complexity of a finite space}

A finite simplicial complex can be associated to any finite space, called the \emph{order complex}.

\begin{definition} \label{KF}
Let $P$ be a finite space. The \emph{order complex} $\K(P)$ is the simplicial complex whose $n$--simplices are linearly ordered subsets of $P$.
Its geometric realization is denoted by $|\K(P)|$.
\end{definition}

Let us examine the relationship between topological and combinatorial complexity.

\begin{theorem}\label{TC=CC}
For any finite space $P$, it holds that $\TC(P) = \CC(P)$.
\end{theorem}

\begin{proof}
We first show the inequality $\TC(P) \geq \CC(P)$. 
Assume that $\TC(P)=n$ with an open cover $\{Q_{i}\}_{i=1}^{n}$ of $P \times P$ and a continuous section $Q_{i} \to P^{I}$ for each $i$. It induces a map $I \to {P}^{Q_i}$ by the exponential law. Hence, we obtain a map $J_{m} \to {P}^{Q_i}$ for some $m \geq 0$ by the homotopy theory of finite spaces, and it induces a combinatorial section $Q_{i} \to P^{J_m}$. It implies that $\CC(P) \leq n$.
Let us show the converse inequality. Assume that $\CC(P)=n$ with an open cover $\{Q_{i}\}_{i=1}^{n}$ of $P \times P$ and a continuous section $s_{i} \co Q_{i} \to P^{J_m}$ of $q_m$ for each $i$ and some $m \geq 0$. Let $\alpha_{m} \co [0,m] \cong |\K(J_{m})| \to J_{m}$ denote the canonical map given by McCord \cite{McC66}. This map is defined by 
\[
\alpha_{m}(t)=
\begin{cases}
2k-1 & (t=2k-1), \\
2k & (2k-1 < t < 2k+1), \\
\end{cases}
\]
for $k=0,1,\ldots$. In particular, this map preserves both ends, i.e., $\alpha_{m}(0)=0$ and $\alpha_{m}(m)=m$. Let $\beta \co I \to J_m$ denote the composition of $\alpha_m$ and the canonical $m$--times isomorphism $I =[0,1] \cong [0,m]$. 
This induces $\beta^* \co P^{J_m} \to P^I$, such that the following diagram is commutative:
\[
\xymatrix{
P^{J_m} \ar[dr]_{q_m} \ar[rr]^{\beta^{*}} && P^{I} \ar[dl]^{p} \\
& P \times P.
}
\]
The composition $\beta^* \circ s_i \co Q_i \to P^I$ is a continuous section of the path fibration $p$ for each $i$. Thus, $\TC(P) \leq n$.
\end{proof}

\begin{corollary}
The following basic homotopical properties hold for combinatorial complexity. 
These have been shown in \cite{Far03} as properties for topological complexity.
\begin{itemize}
\item A finite space $P$ is contractible if and only if $\CC(P)=1$.
\item Combinatorial complexity depends only on the homotopy type of finite spaces, i.e., $\CC(P)=\CC(Q)$ if two finite spaces $P$ and $Q$ are homotopy equivalent (in other words, they have isomorphic cores \cite{Sto66}).
\end{itemize}
\end{corollary}

The next corollary follows from combining Theorem \ref{TC=CC} with \cite[Theorem 5]{Far03}. 
Note that Farber proved $\cat(X) \leq \TC(X)$ without requiring $X$ to be Hausdorff.

\begin{corollary}\label{inequalities}
For any finite space $P$, the following inequalities hold:
\[
\cat(P) \leq \TC(P) = \CC(P) \leq \cat(P \times P) \leq \cat(P)^2.
\]
\end{corollary}

\begin{example}\label{height_1}
Let $P$ be a finite space consisting of $n+m$ points $\{x_1, \ldots, x_n, y_1, \ldots, y_m\}$, where $n,m \geq 2$, with the partial order described by the following:
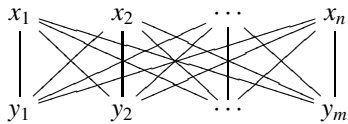
\[
\xymatrix{
x_1 \ar@{-}[d] & x_{2} \ar@{-}[d]& \cdots \ar@{-}[d] & x_n \ar@{-}[d]\\
y_1 \ar@{-}[ru] \ar@{-}[rru] \ar@{-}[rrru]& y_2 \ar@{-}[lu] \ar@{-}[ru] \ar@{-}[rru] & \cdots \ar@{-}[llu] \ar@{-}[lu] \ar@{-}[ru]& y_m \ar@{-}[lllu] \ar@{-}[llu] \ar@{-}[lu]
}
\]
That is, $x_i > y_j$ for any $i,j$. Then we have $\CC(P)=\cat(P \times P)=\cat(P)^2=n^2$.
\end{example}

\begin{proof}
Consider the prime ideal $\{ a \in P\mid a \leq x_i\}$ at each maximal point $x_i$ in $P$.
These are contractible and constitute a cover of $P$, and hence $\cat(P) \leq n$. Corollary \ref{inequalities} shows that $\CC(P) \leq \cat(P)^2 \leq n^2$.
Assume that $\CC(P)=k < n^2$. There exist $k$ open sets covering $P \times P$ with local sections. We can find among them an open set $U$ which contains at least two distinguished maximal points $(x_{i_1},x_{i_2})$ and $(x_{i_3},x_{i_4})$ in $P \times P$. We may assume that $x_{i_2} \neq x_{i_4}$. Fix two distinguished points $y$ and $y'$ in $\{y_1, \ldots, y_m\}$ in $P$.
By Theorem 5 of \cite{Far03}, the open set $V$ in $P$ such that $\{y\} \times V = U \cap (\{y\} \times P)$ is contractible in $P$. However, $V$ is a finite space with height 1 including a loop formed by $x_{i_2} > y < x_{i_4} > y' <x_{i_2}$. This is a contradiction. 
\end{proof}

For a finite space $P$, 
the \emph{opposite space} $P^{\op}$ is the finite space consisting of the same underlying set as $P$ with the reversed partial order of $P$.
Example \ref{height_1} implies the following corollary.

\begin{remark}
In general $\CC(P) \neq \CC(P^{\op})$.
\end{remark}

We will examine the case of the minimal finite model of a sphere.
Let $\calS^n$ denote the finite space consisting of $2n+2$ points $\{e_+^0, e_{-}^0, \ldots, e^n_+, e^n_-\}$ with the partial order defined by $e^k_{p} < e^{\ell}_{q}$ for $k < \ell$ and  $p,q \in \{+,-\}$.
The realization of the order complex $|\K(\calS^n)|$ is homeomorphic to the sphere $S^n$ with dimension $n$. 
We have seen the combinatorial complexity of $\calS^1$ as the case of $n=m=2$ in Example \ref{height_1}. This argument can be extended to the general case of $\calS^n$.

\begin{example}\label{sphere}
We have $\CC(\calS^n)=\cat(\calS^n \times \calS^n) = \cat(\calS^n)^2=4$ for any $n \geq 1$.
\end{example}

\begin{proof}
Corollary \ref{inequalities} shows that $\CC(\calS^n) \leq \cat(\calS^n)^2 \leq 4$.
We use a similar argument as in the proof of Example \ref{height_1}.
If $\CC(\calS^n)<4$, then there is an open set of $\calS^n \times \calS^n$ containing at least two distinguished maximal points. It yields a contractible open set in $\calS^n$ containing $e^n_+$ and $e^n_-$. This coincides with the entire space $\calS^n$, however, this is not contractible. The contradiction implies that $\CC(\calS^n) = 4$.
\end{proof}

A similar argument can be adapted to the product of $\calS^n$.

\begin{example}\label{product}
We have $\CC(\calS^n \times \calS^n)=\cat(\calS^n \times \calS^n)^2=\cat(\calS^n)^4=16$ for any $n \geq 1$.
\end{example}

\begin{remark}\label{product_fail}
The product inequality for the topological complexity is a useful tool \cite[Theorem 11]{Far03}. It is described as follows:
\[
\TC(X \times Y) \leq \TC(X) + \TC(Y)-1
\]
for nice Hausdorff spaces $X$ and $Y$. 
Example \ref{product} proves that this inequality does not hold for finite spaces in general.
\end{remark}

Let us consider another combinatorial model of a circle.

\begin{example}\label{dDelta2}
Let $P$ be a finite space consisting of six points with the following Hasse diagram:
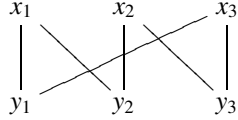
\[
\xymatrix{
x_1 \ar@{-}[rd] & x_2 \ar@{-}[rd]& x_3 \ar@{-}[lld] \\
y_1 \ar@{-}[u] & y_2 \ar@{-}[u] & y_3 \ar@{-}[u]
}
\]
Then we have $\CC(P)=3$.
\end{example}

\begin{proof}
The product $P \times P$ can be described as the face poset of the cell decomposition on torus (see Figure \ref{figure1}), where the opposite sides of the boundary of the maximal square are identified.

	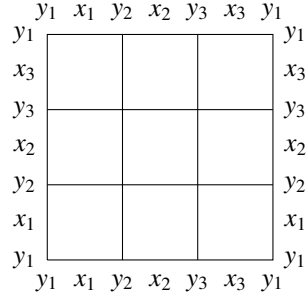
\begin{figure}[ht]
	\begin{center}
   \begin{tikzpicture}
\draw (0,0) -- (3,0) -- (3,3) -- (0,3) -- (0,0);
\draw (1,0) -- (1,3);
\draw (2,0) -- (2,3);
\draw (0,1) -- (3,1);
\draw (0,2) -- (3,2);
\draw (0,-0.3) node {$y_1$};
\draw (0.5,-0.3) node {$x_1$};
\draw (1,-0.3) node {$y_2$};
\draw (1.5,-0.3) node {$x_2$};
\draw (2,-0.3) node {$y_3$};
\draw (2.5,-0.3) node {$x_3$};
\draw (3,-0.3) node {$y_1$};
\draw (-0.3,0) node {$y_1$};
\draw (-0.3,0.5) node {$x_1$};
\draw (-0.3,1) node {$y_2$};
\draw (-0.3,1.5) node {$x_2$};
\draw (-0.3,2) node {$y_3$};
\draw (-0.3,2.5) node {$x_3$};
\draw (-0.3,3) node {$y_1$};

\draw (0,3.3) node {$y_1$};
\draw (0.5,3.3) node {$x_1$};
\draw (1,3.3) node {$y_2$};
\draw (1.5,3.3) node {$x_2$};
\draw (2,3.3) node {$y_3$};
\draw (2.5,3.3) node {$x_3$};
\draw (3,3.3) node {$y_1$};
\draw (3.3,0) node {$y_1$};
\draw (3.3,0.5) node {$x_1$};
\draw (3.3,1) node {$y_2$};
\draw (3.3,1.5) node {$x_2$};
\draw (3.3,2) node {$y_3$};
\draw (3.3,2.5) node {$x_3$};
\draw (3.3,3) node {$y_1$};
  \end{tikzpicture}  
  \end{center}
  \caption{Cell decomposition on torus}
  \label{figure1}
 \end{figure}
 
We can find three collapsible subcomplexes (see Figure \ref{figure2}). In general this does not imply that their face posets are contractible. But in this case they are.
It implies that $\CC(P) \leq \cat(P \times P) \leq 3$ by Corollary \ref{inequalities}.

	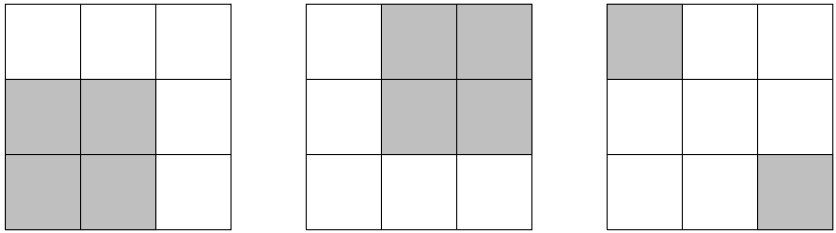
\begin{figure}[ht]
	\begin{center}
   \begin{tikzpicture}
	\draw [fill,lightgray] (0,0) -- (2,0) -- (2,2) -- (0,2) --(0,0);
	\draw (0,0) -- (3,0) -- (3,3) -- (0,3) -- (0,0);
	\draw (1,0) -- (1,3);
	\draw (2,0) -- (2,3);
	\draw (0,1) -- (3,1);
	\draw (0,2) -- (3,2);

	\draw [fill,lightgray] (5,1) -- (7,1) -- (7,3) -- (5,3) --(5,1);
	\draw (4,0) -- (7,0) -- (7,3) -- (4,3) -- (4,0);
	\draw (5,0) -- (5,3);
	\draw (6,0) -- (6,3);
	\draw (4,1) -- (7,1);
	\draw (4,2) -- (7,2);

	\draw [fill,lightgray] (8,2) -- (8,3) -- (9,3) -- (9,2) --(8,2);
	\draw [fill,lightgray] (10,0) -- (11,0) -- (11,1) -- (10,1) --(10,0);
	\draw (8,0) -- (11,0) -- (11,3) -- (8,3) -- (8,0);
	\draw (9,0) -- (9,3);
	\draw (10,0) -- (10,3);
	\draw (8,1) -- (11,1);
	\draw (8,2) -- (11,2);

  \end{tikzpicture}  
  \end{center}
  \caption{Three collapsible subcomplexes covering torus}
  \label{figure2}
 \end{figure}

The space $P$ is not contractible itself, hence $\CC(P)=2$ or $3$.
Let us assume that $\CC(P)=2$ with two open sets covering $P \times P$ and admitting continuous sections. 
Either one $U$ must contains at least five maximal points. This can be described as the face poset of a subcomplex $\mathcal{U}$ of the torus in Figure \ref{figure1}, which contains at least five 2--cells. We will show that $U$ includes either a horizontal slice
\[
(y_1,z)<(x_1,z)>(y_2,z)<(x_2,z)>(y_3,z)<(x_3,z)>(y_1,z),
\]
or a vertical slice
\[
(z,y_1)<(z,x_1)>(z,y_2)<(z,x_2)>(z,y_3)<(z,x_3)>(z,y_1),
\]
for some $z \in P$. Since $\mathcal{U}$ occupies five out of nine $2$--cells, at least two $2$--cells have a common $1$--cell as the boundary. Without loss of generality, we may assume that $\mathcal{U}$ contains the $2$--cells $(x_1,x_1)$ and $(x_2,x_1)$. If $\mathcal{U}$ contains $(x_3,x_1)$ or $(x_3,x_2)$ or $(x_3,x_3)$, then $U$ includes a horizontal slice. On the other hand, if $\mathcal{U}$ contains three out of the four $2$--cells $(x_1,x_2)$, $(x_2,x_2)$, $(x_1,x_3)$, and $(x_2,x_3)$, then $U$ includes a vertical slice. 
The restriction of a continuous section on $U$ to this loop provides a homotopy between the identity $1_P \co P \to P$ and the constant map $z \co P \to P$ onto an element $z$. However, these must be equal by Theorem 3 in \cite{Sto66}, since $P$ is minimal. From this contradiction, we conclude that $\CC(P)=3$.
\end{proof}

For the above $P$, we can easily check that $\cat(P)=2$. This example satisfies the strict inequality $\cat(P \times P) < \cat(P)^2$, however, $\CC(P) = \cat(P \times P)$.
We have not found a finite space $P$ satisfying the strict inequality $\CC(P) < \cat(P \times P)$.

\begin{conjecture}
There exists a finite space $P$ satisfying the strict inequality $\CC(P) < \cat(P \times P)$.
\end{conjecture}


\section{Combinatorial complexity and barycentric subdivision}

As we have seen in the previous section, all examples $\CC(P)$ attain to the upper bound $\cat(P \times P)$ and they do not give good estimates of $\TC(|\K(P)|)$. This results from a small amount of open sets of $P \times P$ compared with $|\K(P)| \times |\K(P)|$. To fix the problem, we extend the idea of $\CC(P)$ using the barycentric subdivision.

\begin{definition}
For a finite space $P$, the barycentric subdivision $\sd(P)$ of $P$ is defined as the face poset $\chi(\K(P))$ of the order complex $\K(P)$. In other words, $\sd(P)$ consists of sequences of ordered elements in $P$ with the subsequence order.
\end{definition}

Let $\tau_P \co \sd(P) \to P$ be the canonical map sending $p_0< \cdots < p_n$ to the last element $p_n$. This is a weak homotopy equivalence \cite{HV93}, and the induced simplicial map $\K(\tau_P) \co \K(\sd(P))=\sd(\K(P)) \to \K(P)$ is a simplicial approximation of the identity on $|\K(P)|$.
For $k \geq 0$, we denote $\tau_P^k \co \sd^k(P) \to P$ as the composition
\[
\sd^k(P) \stackrel{\tau_{\sd^{k-1}(P)}}{\ra} \sd^{k-1}(P) \stackrel{\tau_{\sd^{k-2}(P)}}{\ra} \cdots \stackrel{\tau_{\sd(P)}}{\ra} \sd(P) \stackrel{\tau_{P}}{\ra} P.
\]

\begin{definition}
Let $k \geq 0$ and let $P$ be a finite space. We define $\CC^k(P)$ as the smallest non--negative integer $n$ such that there exists an open cover $\{Q_i\}_{i=1}^{n}$ of $\sd^{k}(P \times P)$ with a map $s_i \co Q_i \to P^{J_m}$ such that $q_m \circ s_i = \tau^k_{P \times P}$ on $Q_i$ for each $i$ and some $m \geq 0$. We will call the maps $s_i$ local sections although they are not rigorously sections.
\end{definition}

Obviously, $\CC(P)=\CC^{0}(P)$ by the definition above.

\begin{lemma}
For any finite space $P$ and $k \geq 0$, it holds that $\CC^{k+1}(P) \leq \CC^{k}(P)$.
\end{lemma}

\begin{proof}
If $\CC^k(P)=n$, then we have an open cover $\{Q_i\}_{i=1}^{n}$ of $\sd^{k}(P \times P)$ with local sections $s_i \co Q_i \to P^{J_m}$. The open set $U_i=\tau^{-1}_{\sd^k(P \times P)}(Q_i)$ of $\sd^{k+1}(P \times P)$ has a local section $s_i \circ \tau_{\sd^k(P \times P)}$ by the following commutative diagram:
\[
\xymatrix{
U_i \ar[r]^-{\tau_{\sd^k(P \times P)}} \ar@{^{(}->}[d] & Q_i \ar[r]^-{s_i}\ar@{^{(}->}[d] & P^{J_m} \ar[d]^{q_m} \\
\sd^{k+1}(P \times P) \ar[r]_-{\tau_{\sd^k(P \times P)}} & \sd^k(P \times P) \ar[r]_-{\tau^k_{P \times P}}& P \times P.
}
\]
The family $\{U_i\}_{i=1}^n$ covers $\sd^{k+1}(P \times P)$, and $\CC^{k+1}(P) \leq n$. 
\end{proof}

The barycentric subdivision gives rise to a functor on the category of finite spaces.
For a continuous map $f  \co P \to Q$ between finite spaces $P$ and $Q$, the induced map $\sd(f) \co  \sd(P) \to \sd(Q)$ is given by $\sd(f)(S) = f(S)$ for a linearly ordered subset
\[
S=\{p_0,p_1,\ldots,p_n \mid p_0 < p_1 < \cdots <p_n\}
\]
of $P$. The map $\tau$ becomes a natural transformation from $\sd$ to the identity functor on finite spaces.
Moreover, the $k$th barycentric subdivision $\sd^k$ is also a functor and $\tau^k$ is a natural transformation from $\sd^k$ to the identity.
For two finite space $P$ and $Q$, we have the canonical map $\varphi \co \sd^k(P \times Q) \to \sd^k(P) \times \sd^k(Q)$ induced by the projections. By the naturality of $\tau^k$ and the universality of products, the following diagram is commutative:
\[
\xymatrix{
\sd^k(P \times Q) \ar[rd]_{\tau^k_{P \times Q}} \ar[rr]^-{\varphi} && \sd^k(P) \times \sd^k(Q) \ar[ld]^{\tau^k_{P} \times \tau^k_Q}\\
& P \times Q.
}
\]

\begin{lemma}\label{CCk<CCsdk}
For any finite space $P$ and $k \geq 0$, it holds that $\CC^{k}(P) \leq \CC(\sd^{k}(P))$.
\end{lemma}

\begin{proof}
If $\CC(\sd^k(P))=n$, then we have an open cover $\{Q_i\}_{i=1}^{n}$ of $\sd^k(P) \times \sd^k(P)$ with local sections $s_i \co Q_i \to \sd^{k}(P)^{J_m}$. The open cover $\{\varphi^{-1}(Q_i)\}_{i=1}^{n}$ of $\sd^{k}(P \times P)$ has local sections $(\tau^k_P)_* \circ s_i \circ \varphi$ by the following commutative diagram:
\[
\xymatrix{
\varphi^{-1}(Q_i) \ar[r]^-{\varphi} \ar@{^{(}->}[d] & Q_i \ar[r]^-{s_i}\ar@{^{(}->}[d] & \sd^k(P)^{J_m} \ar[r]^-{(\tau^k_P)_*} \ar[d]^{q_m}& P^{J_m} \ar[d]^{q_m} \\
\sd^{k}(P \times P) \ar[r]^-{\varphi} \ar@/_3ex/[rrr]_{\tau^k_{P \times P}}& \sd^k(P) \times \sd^k(P) \ar@{=}[r]& \sd^k(P) \times \sd^k(P) \ar[r]^-{\tau^k_P \times \tau^k_P}&P \times P,
}
\]
where $(\tau^k_P)_*$ is the canonical induced map on combinatorial path spaces by $\tau^k_P$.
Thus, $\CC^{k}(P) \leq n$.
\end{proof}

\begin{definition}
Let $P$ be a finite space. We define $\CC^{\infty}(P)$ as the minimum of $\CC^k(P)$:
\[
\CC^{\infty}(P)= \lim_{k \to \infty} \CC^{k}(P) = \min_{k \geq 0}\{\CC^{k}(P)\}.
\]
\end{definition}

\begin{remark}
For a finite space $P$ and $k \geq 0$, let $\rho_j \co \sd^{k}(P \times P) \to P$ denote the composition of $\tau^k_{P \times P} \co \sd^k(P \times P) \to P \times P$ and the $j$th projection for $j=1,2$. We have $\CC^{\infty}(P) \leq n$ if and only if there exist $k \geq 0$ and an open cover $\{Q_i\}_{i=1}^{n}$ of $\sd^{k}(P \times P)$ and a homotopy between the restrictions $\rho_1 \simeq \rho_2 \co Q_i \to P$ for each $i$.
\end{remark}

Now we examine the relation between $\CC^{\infty}(P)$ and the simplicial complexity of the order complex $\K(P)$. We first recall Gonz\'alez's original idea of simplicial complexity for simplicial complexes \cite{Gon}.

For a finite simplicial complex $K$, the barycentric subdivision $\sd(K)$ is isomorphic to the order complex $\K(\chi(K))$ of the face poset of $K$. Fix an order on the set of vertices $V$ of $K$ and consider the product $K \times K$ in the category of ordered simplicial complexes. 
Here, the set of vertices of $K \times K$ is $V \times V$ and a simplex of $K \times K$ is a subset
\[
S = \{(v_0,u_0), (v_1,v_1), \ldots, (v_n,u_n) \mid v_0 \leq v_1 \leq \cdots \leq v_n, u_0 \leq u_1 \leq \cdots \leq u_n\}
\]
of $V \times V$ such that both $\mathrm{pr}_1(S)$ and $\mathrm{pr}_2(S)$ are simplices of $K$, where $\mathrm{pr}_j$ is the projection on $V \times V$ on each factor for $j=1,2$. We choose a simplicial approximation $\iota_{\sd^{k}(K \times K)} \co \sd^{k+1}(K \times K) \to \sd^{k}(K \times K)$ of the identity on $|K| \times |K|$ for $k \geq 0$. 
We denote $\iota_{K \times K}^{k} \co \sd^k(K \times K) \to K \times K$ as the composition
\[
\sd^{k}(K \times K) \stackrel{\iota_{\sd^{k-1}(K \times K)}}{\ra} \sd^{k-1}(K \times K) \stackrel{\iota_{\sd^{k-2}(K \times K)}}{\ra} \cdots \stackrel{\iota_{\sd(K \times K)}}{\ra} \sd(K \times K) \stackrel{\iota_{K \times K}}{\ra} K \times K.
\]
Let $\pi_j \co \sd^{k}(K \times K) \to K$ denote the composition of $\iota^k_{K \times K} \co \sd^k(K \times K) \to K \times K$ and the $j$th projection for $j=1,2$.

\begin{definition}\label{simplicial_complexity}
Let $K$ be a finite simplicial complex. We define $\SC^{k}(K)$ as the smallest non--negative integer $n$ such that there exist subcomplexes $\{L_{i}\}_{i=1}^{n}$ covering $\sd^{k}(K \times K)$ and the restrictions $\pi_1, \pi_2 \co L_i \to K$ lie in the same contiguity class (see \cite{Spa66}) for each $i$. 
\end{definition}

Note that Gonz\'alez used the reduced version, which is one less than the above definition. We have the following decreasing sequence of numbers for a simplicial complex $K$:
\[
\SC^{0}(K) \geq \SC^{1}(K) \geq \SC^2(K) \geq \cdots\geq 0 .
\]

\begin{definition}[Definition 2.5 in \cite{Gon}]
For a finite simplicial complex $K$, the {\em simplicial complexity} $\SC(K)$ is defined as the minimum of $\SC^k(K)$:
\[
\SC(K)=\lim_{k \to \infty}\SC^{k}(K) = \min_{k \geq 0}\{\SC^{k}(K)\}.
\]
\end{definition}

Note that the above definition of simplicial complexity does not depend on the choice of approximations and ordering of vertices. 

We are interested in the case that $K=\K(P)$ for a finite space $P$.
We can choose a linear extension on $P$ (total order compatible with the partial order on $P$), and it numbers the vertices of $\K(P)$.
The product simplicial complex of the two copies of $\K(P)$ coincides with $\K(P \times P)$ (in the category of ordered simplicial complexes). For $k \geq 0$, the induced map
\[
\K\left(\tau_{\sd^k(P \times P)} \right) \co \sd^{k+1}(\K (P \times P))=\K(\sd^{k+1}(P \times P)) \to \K(\sd^k(P \times P)) = \sd^{k} (\K(P \times P))
\]
is an approximation of the identity of $|\K(P)| \times |\K(P)|$. Furthermore, we can choose
\[
\pi_j \co \sd^k(\K(P \times P)) = \K(\sd^k(P \times P)) \ra \K(P)
\]
as $\K(\pr_j \circ \tau^k_{P \times P})$, where $\pr_j$ is the projection of $P \times P$ on each factor for $j=1,2$. This will be used in the next proof.

\begin{theorem}\label{CC=SC}
For any finite space $P$, it holds that $\CC^{\infty}(P) = \SC(\K (P))$.
\end{theorem}

\begin{proof}
We assume that $\CC^{\infty}(P)=n$ with open sets $\{Q_i\}_{i=1}^{n}$ covering $\sd^{k}(P \times P)$ for some $k \geq 0$ and a homotopy $\rho_1 \simeq \rho_2 \co Q_i \to P$ for each $i$.
Proposition 4.11 of \cite{BM12} implies that $\K(\rho_1), \K(\rho_2) \co \K(Q_i) \to \K(P)$ lie in the same contiguity class. The subcomplexes $\K(Q_i)$ constitute a cover of  $\K(\sd^{k}(P \times P)) = \sd^{k}(\K(P \times P))$, and 
\[
\K(\rho_j)=\K(\pr_j \circ \tau^k_{P \times P})=\pi_j
\] 
for $j=1,2$. Thus, $\SC^k(\K (P)) \leq n$ and then $\SC(\K (P)) \leq n$.

Conversely, assume that $\SC(\K (P))=n$. Then $\SC^k(\K (P))=n$ for some $k \geq 0$ and a linear ordering of the vertices of $\K(P)$ extending the order of $P$. Let $\{L_i\}_{i=1}^{n}$ be a covering of $\sd^{k}(\K(P \times P))$ and the restrictions $\pi_1, \pi_2 \co  L_i \to \K(P)$ lie in the same contiguity class for each $i$.
Proposition 4.12 of \cite{BM12} implies that $\chi(\pi_1)$ and $\chi(\pi_2)$ are homotopic. Moreover, $\tau_{P} \circ \chi(\pi_1)$ and $\tau_{P} \circ \chi(\pi_2)$ are homotopic.
The subsets $\chi(L_i)$ constitute an open cover of $\chi(\sd^{k}(\K(P \times P))) = \sd^{k+1}(P \times P)$.
The naturality of $\tau$ makes the following diagram commute:
\[
\xymatrix{
\sd^{k+1}(P \times P) \ar[r]^-{\sd(\tau^k_{P \times P})} \ar[d]_{\tau_{\sd^k(P \times P)}}& \sd(P \times P) \ar[r]^-{\sd(\pr_j)} \ar[d]_{\tau_{P \times P}}& \sd(P) \ar[d]^{\tau_{P}} \\
\sd^k(P \times P) \ar[r]_-{\tau^k_{P \times P}} & P \times P \ar[r]_-{\pr_j} & P.
}
\]
We have
\begin{align*}
\tau_{P} \circ \chi(\pi_j)&=\tau_{P} \circ \chi \left(\K \left(\pr_j \circ \tau^k_{P \times P} \right) \right) \\
&=\tau_{P} \circ \sd (\pr_j \circ \tau^k_{P \times P} ) \\
&= \tau_{P} \circ \sd(\pr_j) \circ \sd(\tau^k_{P \times P}) \\
&= \pr_j \circ \tau^{k+1}_{P \times P}=\rho_j
\end{align*}
for $j=1,2$. Thus,  $\CC^{k+1}(P) \leq n$ and then $\CC^{\infty}(P) \leq n$.
\end{proof}

The next corollary follows from Gonz\'alez's result \cite[Theorem 2.6]{Gon}.

\begin{corollary}\label{CC=TC}
For any finite space $P$, it holds that $\CC^{\infty}(P)=\TC(|\K(P)|)$.
\end{corollary}

The above corollary implies that the topological complexity of the geometric realization of the order complex of a finite space $P$ can be computed in combinatorial terms of $P$. As a result, we have the following relation for a finite space $P$:
\[
\TC(P)=\CC(P) \geq \CC^1(P) \geq \CC^2(P) \geq \cdots \geq \CC^{\infty}(P)=\SC(\K(P)) = \TC(|\K(P)|).
\]
For the face poset of a simplicial complex, a similar result to Theorem \ref{CC=SC} holds.

\begin{proposition}
For any finite simplicial complex $K$, it holds that $\CC^{\infty}(\chi(K))=\SC(K)=\TC(|K|)$.
\end{proposition}

\begin{proof}
Theorem \ref{CC=SC} and \cite[Theorem 2.6]{Gon} show that 
\[
\CC^{\infty}(\chi(K)) = \SC(\sd(K))=\TC(|\sd(K)|)=\TC(|K|) = \SC(K).
\]
\end{proof}

Let us focus on the properties of $\CC^{\infty}(P)$.

\begin{proposition}\label{h_properties}
Let $P$ and $Q$ be finite spaces.
\begin{enumerate}
\item $\CC^{\infty}(P)=1$ if and only if $P$ is weakly contractible.
\item $\CC^{\infty}(P)=\CC^{\infty}(Q)$ if $P$ and $Q$ are weakly homotopy equivalent.
\end{enumerate}
\end{proposition}

\begin{proof}
(1) By McCord's weak homotopy equivalence $|\K(P)| \to P$ (see \cite{McC66}), $P$ is weakly contractible if and only if $|\K(P)|$ is contractible. The result follows from Corollary \ref{CC=TC}. 

(2) If two finite spaces $P$ and $Q$ are weakly homotopy equivalent, then $|\K(P)|$ and $|\K(Q)|$ are homotopy equivalent by McCord's weak homotopy equivalence. It implies the following equality by Corollary \ref{CC=TC}:
\[
\CC^{\infty}(P) = \TC(|\K(P)|) = \TC(|\K(Q)|) = \CC^{\infty}(Q).
\]
\end{proof}

\begin{proposition}
For finite spaces $P$ and $Q$, the following product inequality holds:
\[
\CC^{\infty}(P \times Q) \leq \CC^{\infty}(P) + \CC^{\infty}(Q) -1.
\]
\end{proposition}

\begin{proof}
By Corollary \ref{CC=TC} and the product inequality for the topological complexity,
we have
\begin{align*}
\CC^{\infty}(P \times Q) &= \TC(|\K(P \times Q)|) \\
&= \TC(|\K(P)| \times |\K(Q)|) \\
&\leq \TC(|\K(P)|)+\TC(|\K(Q)|)-1 \\
&= \CC^{\infty}(P) + \CC^{\infty}(Q) -1.
\end{align*}
\end{proof}

The following proposition follows from the fact that $|\K(P)| \cong |\K(P^{\op})|$ for any finite space $P$.

\begin{proposition}
For any finite space $P$, it holds that $\CC^{\infty}(P)=\CC^{\infty}(P^{\op})$.
\end{proposition}

Let us compute the case $\calS^1$ in reference to \cite[Section 3]{Gon}.

\begin{example}
We have $\CC^{\infty}(\calS^1)=2$.
\end{example}

\begin{proof}
It suffices to show that $\CC^{2}(\calS^1)=2$. The finite space $\sd^2(\calS^1 \times \calS^1)$ is the face poset of the simplicial complex of torus $T^2$ shown in Figure \ref{figure3}, where the opposite sides of the boundary of the maximal square are identified.
	
	\begin{figure}[ht]
	\begin{center}
   \begin{tikzpicture}
\draw [fill,lightgray] (1,0) -- (4,3) -- (4,4) -- (3,4) -- (0,1)--(0,0);
\draw [fill,lightgray] (3,0) -- (4,1) -- (4,0) -- (3,0);
\draw [fill,lightgray] (0,3) -- (1,4) -- (0,4) -- (0,3);

\draw (0,0) -- (4,0) -- (4,4) -- (0,4) -- (0,0);
\draw (0,1) -- (4,1);
\draw (0,2) -- (4,2);
\draw (0,3) -- (4,3);
\draw (1,0) -- (1,4);
\draw (2,0) -- (2,4);
\draw (3,0) -- (3,4);
	
\draw (0,0) -- (4,4);
\draw (1,0) -- (4,3);
\draw (0,1) -- (3,4);
\draw (0,3) -- (1,4);
\draw (3,0) -- (4,1);

\begin{scope}[very thick]
\draw (0,2) -- (2,4);
\draw (2,0) -- (4,2);
\end{scope}

\draw (0,4) -- (4,0);
\draw (1,4) -- (4,1);
\draw (2,4) -- (4,2);
\draw (3,4) -- (4,3);
\draw (0,3) -- (3,0);
\draw (0,2) -- (2,0);
\draw (0,1) -- (1,0);

\draw (0,0) -- (0.5,1);
\draw (0,0.5) -- (1,1);
\draw (1,0) -- (0,1);
\draw (0.5,0) -- (1,1);
\draw (0,0) -- (1,0.5);
\draw (1,0.5) -- (2,0);
\draw (1,1) -- (1.5,0);
\draw (1,1) -- (2,0.5);
\draw (2,0) -- (1.5,1);
\draw (0.5,1) -- (0,2);
\draw (0,1.5) -- (1,1);
\draw (1,1) -- (0.5,2);
\draw (0,2) -- (1,1.5);
\draw (1.5,1) -- (2,2);
\draw (2,1.5) -- (1,1);
\draw (1,1) -- (1.5,2);
\draw (2,2) -- (1,1.5);

\draw (2,0) -- (2.5,1);
\draw (2,0.5) -- (3,1);
\draw (3,0) -- (2,1);
\draw (2.5,0) -- (3,1);
\draw (2,0) -- (3,0.5);
\draw (3,0.5) -- (4,0);
\draw (3,1) -- (3.5,0);
\draw (3,1) -- (4,0.5);
\draw (4,0) -- (3.5,1);
\draw (2.5,1) -- (2,2);
\draw (2,1.5) -- (3,1);
\draw (3,1) -- (2.5,2);
\draw (2,2) -- (3,1.5);
\draw (3.5,1) -- (4,2);
\draw (4,1.5) -- (3,1);
\draw (3,1) -- (3.5,2);
\draw (4,2) -- (3,1.5);

\draw (0,2) -- (0.5,3);
\draw (0,2.5) -- (1,3);
\draw (1,2) -- (0,3);
\draw (0.5,2) -- (1,3);
\draw (0,2) -- (1,2.5);
\draw (1,2.5) -- (2,2);
\draw (1,3) -- (1.5,2);
\draw (0,3.5) -- (2,2.5);
\draw (2,2) -- (1.5,3);
\draw (0.5,3) -- (0,4);
\draw (0,2.5) -- (1,3);
\draw (1,3) -- (0.5,4);
\draw (0,4) -- (1,3.5);
\draw (1.5,3) -- (2,4);
\draw (2,3.5) -- (1,3);
\draw (1,3) -- (1.5,4);
\draw (2,4) -- (1,3.5);

\draw (2,2) -- (2.5,3);
\draw (2,2.5) -- (3,3);
\draw (3,2) -- (2,3);
\draw (2.5,2) -- (3,3);
\draw (2,2) -- (3,2.5);
\draw (3,2.5) -- (4,2);
\draw (3,3) -- (3.5,2);
\draw (3,3) -- (4,2.5);
\draw (4,2) -- (3.5,3);
\draw (2.5,3) -- (2,4);
\draw (2,3.5) -- (3,3);
\draw (3,3) -- (2.5,4);
\draw (2,4) -- (3,3.5);
\draw (3.5,3) -- (4,4);
\draw (4,3.5) -- (3,3);
\draw (3,3) -- (3.5,4);
\draw (4,4) -- (3,3.5);

\draw (0,-0.3) node {$e^0_+$};
\draw (1,-0.3) node {$e^1_+$};
\draw (2,-0.3) node {$e^0_-$};
\draw (3,-0.3) node {$e^1_-$};
\draw (4,-0.3) node {$e^0_+$};

\draw (-0.3,0) node {$e^0_+$};
\draw (-0.3,1) node {$e^1_+$};
\draw (-0.3,2) node {$e^0_-$};
\draw (-0.3,3) node {$e^1_-$};
\draw (-0.3,4) node {$e^0_+$};

\draw (0,4.3) node {$e^0_+$};
\draw (1,4.3) node {$e^1_+$};
\draw (2,4.3) node {$e^0_-$};
\draw (3,4.3) node {$e^1_-$};
\draw (4,4.3) node {$e^0_+$};

\draw (4.3,0) node {$e^0_+$};
\draw (4.3,1) node {$e^1_+$};
\draw (4.3,2) node {$e^0_-$};
\draw (4.3,3) node {$e^1_-$};
\draw (4.3,4) node {$e^0_+$};
  \end{tikzpicture}  
  \end{center}
  \caption{Simplicial subdivision of the torus}
  \label{figure3}
 \end{figure}
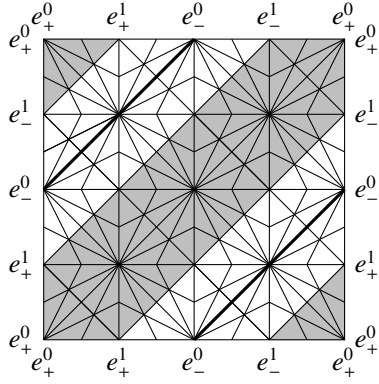

We can take two subcomplexes, shaded $K_1$ and unshaded $K_2$ in Figure \ref{figure3}. 
We notice that $K_1$ can be strongly collapsed \cite{BM12} onto the subcomplex of diagonal $\Delta=\{(x,x)\}$ in $T^2$. The face poset $\chi(K_1)$ is an open set of $\sd^2(\calS^1 \times \calS^1)$, and $\chi(\Delta)$ is a deformation retract of $\chi(K_1)$. The maps $\rho_1, \rho_2 \co \chi(\Delta) \to \calS^1$ are equal, and hence homotopic.

Similarly, $K_2$ can be strongly collapsed onto the core subcomplex $\nabla$ described as the thick line in Figure \ref{figure3}.
The face poset $\chi(\nabla)$ is a deformation retract of $\chi(K_2)$, and 
\[
\rho_j \co \chi(\nabla) \to \calS^1
\]
can be described as follows: 
\[
\rho_1(t) = \begin{cases}
e_+^{0}, & t=1, \\
e_+^{1}, & 2 \leq t \leq 8, \\
e_-^{0}, & t=9, \\
e_-^{1}, & 10 \leq t \leq 16,
\end{cases} \quad 
\rho_2(t) = \begin{cases}
e_-^{0}, & t=1, \\
e_-^{1}, & 2 \leq t \leq 8, \\
e_+^{0}, & t=9, \\
e_+^{1}, & 10 \leq t \leq 16,
\end{cases} 
\]
where we regard $\chi(\nabla)$ as a finite space with $16$ points formed by $1<2>3<\cdots <16>1$.
Moreover, we consider the following maps $f_i \co \chi(\nabla) \to \calS^1$ as described in Table \ref{f}.

\begin{table}[htb]
\begin{center}
  \begin{tabular}{|l||c|c|c|c|c|c|c|c|c|c|c|c|c|c|c|c|} \hline
     & 1&2 & 3 & 4 & 5& 6 & 7 & 8 & 9&10&11&12 &13&14&15&16\\ \hline 
    $\rho_1$ & $e^0_+$& $e^1_+$ & $e^1_+$ & $e^1_+$ & $e^1_+$ & $e^1_+$ & $e^1_+$& $e^1_+$ & $e^0_-$ & $e^1_-$ & $e^1_-$ & $e^1_-$ & $e^1_-$ & $e^1_-$ & $e^1_-$ & $e^1_-$ \\ \hline
    $f_1$ & $e^0_+$& $e^0_+$ & $e^0_+$ & $e^1_+$ & $e^0_-$ & $e^0_-$ & $e^0_-$& $e^0_-$ & $e^0_-$ & $e^0_-$ & $e^0_-$ & $e^1_-$ & $e^0_+$ & $e^0_+$ & $e^0_+$ & $e^0_+$ \\ \hline
        $f_2$ & $e^1_+$& $e^1_+$ & $e^1_+$ & $e^1_+$ & $e^0_-$ & $e^1_-$ & $e^1_-$& $e^1_-$ & $e^1_-$ & $e^1_-$ & $e^1_-$ & $e^1_-$ & $e^0_+$ & $e^1_+$ & $e^1_+$ & $e^1_+$ \\ \hline
        $f_3$ & $e^0_-$& $e^0_-$ & $e^0_-$ & $e^0_-$ & $e^0_-$ & $e^1_-$ & $e^0_+$& $e^0_+$ & $e^0_+$ & $e^0_+$ & $e^0_+$ & $e^0_+$ & $e^0_+$ & $e^1_+$ & $e^0_-$ & $e^0_-$ \\ \hline   
        $\rho_2$ & $e^0_-$& $e^1_-$ & $e^1_-$ & $e^1_-$ & $e^1_-$ & $e^1_-$ & $e^1_-$ & $e^1_-$ & $e^0_+$ & $e^1_+$ & $e^1_+$ & $e^1_+$ & $e^1_+$ & $e^1_+$ & $e^1_+$ & $e^1_+$ \\ \hline           
  \end{tabular}
\end{center}
\caption{The values of $f_i$ and $\rho_j$}
\label{f}
\end{table}

We notice that $\rho_1>f_1<f_2>f_3<\rho_2$, and these are homotopic. Hence, we have continuous sections on $\chi(K_1)$ and $\chi(K_2)$, respectively, and $\CC^{2}(\calS^1)=2$.
\end{proof}

For a finite space $P$, the inequality $\CC^k(P) \leq \CC(\sd^k(P))$ in Lemma \ref{CCk<CCsdk} implies that
\[
\CC^{\infty}(P) \leq \min_{k \geq 0} \{ \CC(\sd^k(P))\}.
\]
Both these inequalities can be strict. 
For example, if $P$ is weakly contractible and non--contractible, then $\sd^k(P)$ is not contractible for every $k \geq 0$ \cite{BM12}. Therefore, $\CC(\sd^k(P)) \geq 2$ for any $k \geq 0$. 
However, $\CC^{\infty}(P) =1$ by Proposition \ref{h_properties}.

In Example \ref{dDelta2}, we have seen a finite model $P$ of a circle such that
$\CC^{\infty}(P) =\TC(S^1)=2$ and $\CC(P)=3$. We do not know if $\CC(\sd^k(P))=3$ for every $k \geq 0$.

\begin{conjecture}
Let $P$ be a finite model of a circle in Example \ref{dDelta2}. Then $\CC(\sd^k(P))=3$ for every $k \geq 0$.
\end{conjecture}

\end{document}